\documentclass[12pt, reqno]{amsart}
\usepackage{amsmath, amsthm, amscd, amsfonts, amssymb, graphicx, color}
\usepackage[bookmarksnumbered, colorlinks, plainpages]{hyperref}

\newtheorem{theorem}{Theorem}[section]

\newtheorem{corollary}[theorem]{Corollary}
\theoremstyle{definition}
\newtheorem{definition}[theorem]{Definition}
\newtheorem{example}[theorem]{Example}

\theoremstyle{remark}

\numberwithin{equation}{section}

\begin{document}
\setcounter{page}{1}


\title[Stabilities and non-stabilities of the reciprocal-nonic and $\cdots$]{Stabilities and non-stabilities of the reciprocal-nonic and the reciprocal-decic functional equations}

\author[A. Bodaghi, B. V. Senthil Kumar  and J. M. Rassias]{Abasalt Bodaghi, Beri Venkatachalapathy Senthil Kumar  and John Michael Rassias}

\address{Department of Mathematics, Garmsar Branch, Islamic Azad University, Garmsar, Iran.}
\email{\textcolor[rgb]{0.00,0.00,0.84}{abasalt.bodaghi@gmail.com}}
\address{Department of Mathematics, C. Abdul Hakeem College of Engineering \& Technology, Melvisharam - 632 509, Tamil Nadu, India.}
\email{\textcolor[rgb]{0.00,0.00,0.84}{bvskumarmaths@gmail.com}}
\address{Pedagogical Department E.E., Section of Mathematics and Informatics, National and Capodistrian University of Athens, 4, Agamemnonos Str., Aghia Paraskevi, Athens, Attikis 15342, Greece.}
\email{\textcolor[rgb]{0.00,0.00,0.84}{jrassias@primedu.uoa.gr}}





\begin{abstract}
This paper focuses at the various stability results of reciprocal-nonic and  reciprocal-decic functional equations in non-Archimedean fields and illustrations of the proper examples for their non-stabilities.
\end{abstract}
\maketitle

\section{Introduction}\label{sec1}
The source for the development of stability of functional equations is the question solicited by Ulam \cite{SMUlam}. Hyers \cite{Hyers} presented an excellent answer to the question of Ulam. Later, the result of Hyers was generalized and refined further by many great mathematicians like Aoki \cite{Aoki}, Th.M. Rassias \cite{ThMRassias78}, J.M. Rassias \cite{JMRassias82(1)} and  G\u{a}vruta \cite{PGavruta} in various directions. The progress of the theory of stability of various types of functional equations such as quadratic, cubic, quartic, quintic, sextic, septic, octic, nonic, decic, undecic, duodecic, tredecic, quattordecic have been dealt by many mathematicians and there are lot of interesting and significant results available in the literature. 
\par For the first time, Ravi and the second author \cite{KRavi08(1)} achieved various stability results of the following functional equation
\begin{equation}\label{e.1.1}
\phi(u+v)=\frac{\phi(u)\phi(v)}{\phi(u)+\phi(v)}
\end{equation}
where $\phi:\mathbb{R^{*}}\longrightarrow \mathbb{R}$ is a mapping and $\mathbb{R}^{*}=\mathbb{R}\backslash\{0\}$. The rational function $\phi(x)=\frac{c}{x}$ is a solution of the functional equation (\ref{e.1.1}). The functional equation (\ref{e.1.1}) is interconnected with  ``Reciprocal formula" which will be useful in any electric circuit with couple of parallel resistors \cite{KRavi12}.  Hence, the equation (\ref{e.1.1}) is said to be reciprocal functional equation. The geometrical interpretation of the equation (\ref{e.1.1}) is also discussed in \cite{KRavi12}. Suitable counter-examples are presented to show the non-stability of the equation (\ref{e.1.1}) controlled by the sum of powers of norms and product of powers of norms for singular cases in \cite{KRavi13(2)} and \cite{KRavi13(3)} respectively. 
\par Senthil Kumar et al. \cite{BVS16(1)} obtained the solutions of the following functional equations (arising from arithmetic and harmonic means of position of pixels in an image)
\begin{align}\label{e.1.2}
f(x,y)=\frac{1}{4}[f(x+t,y+t)+f(x+t,y-t)+f(x-t,y+t)+f(x-t,y-t)]
\end{align}
and
 \begin{align}\label{e.1.3} 
 f(x,y) =\frac{f_{1}(x,y,t)}{f_{2}(x,y,t)}
\end{align}
where $f_1(x,y,t)=4f(x+t,y+t)f(x+t,y-t)f(x-t,y+t)f(x-t,y-t)$ and
\begin{align*}
 f_2(x,y,t)& = f(x+t,y+t)f(x+t,y-t)f(x-t,y+t)\\
&+f(x+t,y+t)f(x+t,y-t)f(x-t,y-t)\\
&+f(x+t,y+t)f(x-t,y+t)f(x-t,y-t)\\
&+f(x+t,y-t)f(x-t,y+t)f(x-t,y-t)\neq 0.
\end{align*}
for all $x,y,t\in \mathbb{N}$. The equations (\ref{e.1.2}) and (\ref{e.1.3}) are applied to remove noise in an image by
filtering techniques. The study of stability of several  functional equations in various spaces and their solutions as rational functions can be found in  \cite{Bodaghi14(1)}, \cite{Bodaghi14(2)}, \cite{Bodaghi16}, \cite{Bodaghi15}, \cite{BVS17(1)}, \cite{bsb},  \cite{SOKim}, \cite{KRavi10(2)}, \cite{KRavi10(4)}, \cite{KRavi11(1)},   \cite{KRavi15}, \cite{KRavi11(3)}, \cite{BVS16(2)}, \cite{BVS16(3)}, \cite{BVS16(4)}.
\par In this study, we consider the following reciprocal-nonic functional equation
\begin{align}\label{1.1}
& n(2x+y)+n(2x-y)\notag\\
&~=\frac{4n(x)n(y)}{\left(4n(y)^{\frac{2}{9}}-n(x)^{\frac{2}{9}}\right)^9}\Big[256n(y)+2304n(x)^{\frac{2}{9}}n(y)^{\frac{7}{9}}+2016n(x)^{\frac{4}{9}}n(y)^{\frac{5}{9}}\notag\\
&\qquad\qquad\qquad\qquad\qquad\qquad\qquad\qquad +336n(x)^{\frac{6}{9}}n(y)^{\frac{3}{9}}+n(x)^{\frac{8}{9}}n(y)^{\frac{1}{9}}\Big]
\end{align}
and the reciprocal-decic functional equation
\begin{align}\label{1.2}
& d(2x+y)+d(2x-y)\notag\\
&~=\frac{2d(x)d(y)}{\left(4d(y)^{\frac{1}{5}}-d(x)^{\frac{1}{5}}\right)^{10}}\Big[1024d(y)+11520d(x)^{\frac{1}{5}}d(y)^{\frac{4}{5}}+13440d(x)^{\frac{2}{5}}d(y)^{\frac{3}{5}}\notag\\
&\qquad\qquad\qquad\qquad\qquad +3360d(x)^{\frac{3}{5}}d(y)^{\frac{2}{5}}+180d(x)^{\frac{4}{5}}d(y)^{\frac{1}{5}}+d(x)\Big].
\end{align}
The reciprocal-nonic function $n(x)=\frac{1}{x^9}$ and the reciprocal-decic function $d(x)=\frac{1}{x^{10}}$ satisfy the equations (\ref{1.1}) and (\ref{1.2}), respectively. Hence the functions $n(x)=\frac{1}{x^9}$ and $d(x)=\frac{1}{x^{10}}$ are solutions of equations (\ref{1.1}) and (\ref{1.2}), respectively. We investigate various Ulam stabilties of the above equations (\ref{1.1}) and (\ref{1.2}) and also prove the non-stability results through proper illustrative examples. 


\section{Preliminaries}

In this section, we recall the basic facts of non-Archimdean fields.
\begin{definition}\label{Def 2.1}
A field $\mathbb{K}$ provided with a function \emph{(}valuation\emph{)} $|\cdot|$ from $\mathbb{K}$ into $[0,\infty)$ is called a non-Archimedean field provided the subsequent conditions hold:
\begin{enumerate}
\item[1.] $|s|=0$ if and only if $s=0$;
\item[2.] $|st|=|s||t|$;
\item[3.] $|s+t|\leq {\text max}\{|s|, |t|\}$, for all $s, t\in \mathbb{K}$. 
\end{enumerate}
\end{definition}
Clearly $|1|=|-1|=1$ and $|n|\leq 1$ for all $n\in \mathbb{N}$. We always assume, in addition, that $|\cdot|$  is non-trivial, i.e., there exists an $\mu_{0}\in \mathbb K$  such that $|\mu_{0}|\neq {0,1}$. A sequence $\{u_n\}$ is Cauchy if and only if $\{u_{n+1}-u_n\}$ converges to zero in a non-Archimedean field because
$$\left|u_k-u_l\right|\leq \text{max}\left\{\left|u_{i+1}-u_i\right|:l\leq i\leq k-1\right\}\qquad (k>l).$$
By a complete non-Archimedean field, we mean that every Cauchy sequence is convergent in the field.

In \cite{hen}, Hensel discovered the $p$-adic numbers as a number theoretical analogue of
power series in complex analysis. The most interesting example of non-Archimedean normed spaces
is $p$-adic numbers. A key property of $p$-adic numbers is that they do not satisfy the Archimedean axiom: for all $x,y>0$, there exists an integer $n$ such that $x<ny$.  Let $p$ be a prime number. For any non-zero rational number $x=p^r\frac{m}{n}$ in which $m$ and $n$ are coprime to the prime number $p$. Consider the $p$-adic absolute value $|x|_p=p^{-r}$ on $\mathbb{Q}$. It is easy to check that $|\cdot|$ is a non-Archimedean norm on $\mathbb{Q}$. The completion of $\mathbb{Q}$ with respect to $|\cdot|$ which is denoted by $\mathbb{Q}_p$ is said to be the $p$-adic number field. Note that if $p>2$, then $\left|2^n\right|=1$ for all integers $n$.
\par Let us presume that throughout this paper, $\mathbb{A}$ and $\mathbb{B}$ are a non-Archimedean field and a complete non-Archimedean field, respectively. In the sequel, we denote $\mathbb{A}^{*}=\mathbb{A}\backslash\{0\}$, where $\mathbb{A}$ is a non-Archimedean field. For the equations ({\ref{1.1}}) and (\ref{1.2}), we define the difference operators $\Delta_1n, \Delta_2d:\mathbb{A}^{*}\times\mathbb{A}^{*}\longrightarrow \mathbb{B}$ through
\begin{align*}
\Delta_1n(x,y) & = n(2x+y)+n(2x-y)\notag\\
&\quad -\frac{4n(x)n(y)}{\left(4n(y)^{\frac{2}{9}}-n(x)^{\frac{2}{9}}\right)^9}\Big[256n(y)+2304n(x)^{\frac{2}{9}}n(y)^{\frac{7}{9}}+2016n(x)^{\frac{4}{9}}n(y)^{\frac{5}{9}}\notag\\
&\qquad\qquad\qquad\qquad\qquad\qquad\qquad\qquad +336n(x)^{\frac{6}{9}}n(y)^{\frac{3}{9}}+n(x)^{\frac{8}{9}}n(y)^{\frac{1}{9}}\Big]
\end{align*}
and
\begin{align*}
\Delta_2d(x,y) & = d(2x+y)+d(2x-y)\notag\\
& \quad -\frac{2d(x)d(y)}{\left(4d(y)^{\frac{1}{5}}-d(x)^{\frac{1}{5}}\right)^{10}}\Big[1024d(y)+11520d(x)^{\frac{1}{5}}d(y)^{\frac{4}{5}}+13440d(x)^{\frac{2}{5}}d(y)^{\frac{3}{5}}\notag\\
&\qquad\qquad\qquad\qquad\qquad +3360d(x)^{\frac{3}{5}}d(y)^{\frac{2}{5}}+180d(x)^{\frac{4}{5}}d(y)^{\frac{1}{5}}+d(x)\Big]
\end{align*}
for all $x,y\in \mathbb{A^{*}}$.


\section{Stability results}
\indent\quad In this section, we investigate the various Ulam stabilities of equations ({\ref{1.1}}) and (\ref{1.2}) in non-Archimedean fields.
\begin{definition}\label{Def 2.3}
A mapping $n:\mathbb{A^{*}}\longrightarrow \mathbb{B}$ is said to be as reciprocal-nonic mapping if $n$ satisfies the equation \emph{({\ref{1.1}})}. Also, a mapping $d:\mathbb{A^{*}}\longrightarrow \mathbb{B}$ is called as reciprocal-decic mapping if $d$ satisfies the equation \emph{({\ref{1.2}})}.
\end{definition}

{\it Assumptions on the above definition and equations \emph{(\ref{1.1})} and \emph{(\ref{1.2})}}. By assumuing $n(x)\neq 0$, $n(y)\neq 0$, $d(x)\neq 0$, $d(y)\neq 0$, $4n(y)^{\frac{2}{9}}-n(x)^{\frac{2}{9}}\neq 0$ and $4d(y)^{\frac{1}{5}}-d(x)^{\frac{1}{5}}\neq 0$ for all $x,y\in \mathbb{A^{*}}$,  the singular cases are eliminated. 

\begin{theorem}\label{Thm 3.1}
Let $p\in\{1,-1\}$. Let $\zeta:\mathbb{A^{*}}\times \mathbb{A^{*}}\longrightarrow [0,\infty)$ be a function such that
\begin{equation}\label{3.1}
\lim \limits_{k\rightarrow \infty}\left|\frac{1}{19683}\right|^{pk} \zeta\left(\frac{x}{3^{pk+\frac{p+1}{2}}},\frac{y}{3^{pk+\frac{p+1}{2}}}\right)=0
\end{equation}
for all $x,y\in \mathbb{A^{*}}$. Suppose that $n:\mathbb{A^{*}}\longrightarrow \mathbb{B}$ is a mapping satisfying the inequality 
\begin{equation}\label{3.2}
\left|\Delta_1n(x,y)\right|\leq \zeta(x,y)
\end{equation}
for all $x,y\in \mathbb{A^{*}}$. Then, there occurs a distinct reciprocal-nonic mapping $\mathcal{N}:\mathbb{A^{*}}\longrightarrow \mathbb{B}$ which satisfies \emph{(\ref{1.1})} and
\begin{equation}\label{3.3}
\left|n(x)-\mathcal{N}(x)\right|\leq \textup{max}\left\{\left|\frac{1}{19683}\right|^{pl+\frac{p-1}{2}}\zeta\left(\frac{x}{3^{pl+\frac{p+1}{2}}},\frac{x}{3^{pk+\frac{p+1}{2}}}\right): l\in \mathbb{N}\cup \{0\}\right\}
\end{equation}
for all $x\in \mathbb{A^{*}}$.
\end{theorem}
\begin{proof}
Switching $(x,y)$ into $(x,x)$ in (\ref{3.2}), we obtain
\begin{equation}\label{3.4}
\left|n(x)-\frac{1}{19683^{p}}n\left(\frac{x}{3^{p}}\right)\right|\leq |19683|^{\frac{|p-1|}{2}}\zeta\left(\frac{x}{3^{\frac{p+1}{2}}},\frac{x}{3^{\frac{p+1}{2}}}\right)
\end{equation}
for all $x\in \mathbb{A^{*}}$. Now, by interchanging $x$ into $\frac{x}{3^{pk}}$ in (\ref{3.5}) and multiplying the resultant by $\left|\frac{1}{19683}\right|^{pk}$, we arrive at
\begin{align}\label{3.5}
& \left|\frac{1}{19683^{pk}}n\left(\frac{x}{3^{pk}}\right)-\frac{1}{19683^{(k+1)p}}n\left(\frac{x}{3^{(k+1)p}}\right)\right|\notag\\
&\qquad\qquad\qquad\qquad\leq \left|\frac{1}{19683}\right|^{pk+\frac{p-1}{2}}\zeta\left(\frac{x}{3^{pk+\frac{p+1}{2}}},\frac{x}{3^{pk+\frac{p+1}{2}}}\right)
\end{align}
for all $x\in \mathbb{A^{*}}$. Using the inequalities (\ref{3.1}) and (\ref{3.5}), we see that the sequence $\left\{\frac{1}{19683^{pk}}n\left(\frac{x}{3^{pk}}\right)\right\}$ is Cauchy. Due to completeness of $\mathbb{B}$, this sequence converges to a mapping $\mathcal{N}:\mathbb{A^{*}}\longrightarrow \mathbb{B}$ defined by
\begin{equation}\label{3.6}
\mathcal{N}(x) = \lim \limits_{k\rightarrow \infty}\frac{1}{19683^{pk}}n\left(\frac{x}{3^{pk}}\right).
\end{equation}
Also, for every $x\in \mathbb{A^{*}}$ and non-negative integers $k$, we find
\begin{align}\label{3.7}
& \left|\frac{1}{19683^{pk}}n\left(\frac{x}{3^{pk}}\right)-n(x)\right|\notag\\
&\qquad\qquad = \left|\sum_{l=0}^{k-1}\left\{\frac{1}{19683^{p(l+1)}}n\left(\frac{x}{3^{p(l+1)}}\right)-\frac{1}{19683^{pl}}n\left(\frac{x}{3^{pl}}\right)\right\}\right|\qquad\qquad\qquad\qquad\qquad\qquad\notag\\
&\qquad\qquad \leq \text{max}\left\{\left|\frac{1}{19683^{p(l+1)}}n\left(\frac{x}{3^{p(l+1)}}\right)-\frac{1}{19683^{pl}}n\left(\frac{x}{3^{pl}}\right)\right|:0\leq l<k\right\}\notag\\
&\qquad\qquad \leq \text{max}\left\{\left|\frac{1}{19683}\right|^{pl+\frac{p-1}{2}}\zeta\left(\frac{x}{3^{pl+\frac{p+1}{2}}},\frac{x}{3^{pl+\frac{p+1}{2}}}\right):0\leq l<k\right\}.
\end{align}
Letting $k\rightarrow \infty$ in the inequality (\ref{3.7}) and using (\ref{3.6}), we attain that the inequality (\ref{3.3}) holds. Once more, by applying the inequalities (\ref{3.1}), (\ref{3.2}) and (\ref{3.6}), for every $x,y\in \mathbb{A^{*}}$, we arrive at
\begin{align*}
\left|\Delta_1n(x,y)\right| & = \lim \limits_{k\rightarrow \infty}\left|\frac{1}{19683}\right|^{pk}\left|\Delta_1n\left(\frac{x}{3^{pk}},\frac{y}{3^{pk}}\right)\right| \leq \lim \limits_{k\rightarrow \infty}\left|\frac{1}{19683}\right|^{pk}\zeta\left(\frac{x}{3^{pk}},\frac{y}{3^{pk}}\right)=0.
\end{align*}
Hence, the mapping $\mathcal{N}$ satisfies (\ref{1.1}) and so it is reciprocal-nonic mapping. Next, we confirm that $\mathcal{N}$ is unique. Let us consider $\mathcal{N}^{\prime}:\mathbb{A^{*}}\longrightarrow \mathbb{B}$ be another reciprocal-nonic mapping satisfying (\ref{3.3}). Then 
\begin{align*}
& \left|\mathcal{N}(x)-\mathcal{N}^{\prime}(x)\right|\notag\\
&= \lim \limits_{m\rightarrow \infty}\left|\frac{1}{19683}\right|^{pm}\left|\mathcal{N}\left(\frac{x}{3^{pm}}\right)-\mathcal{N}^{\prime}\left(\frac{x}{3^{pm}}\right)\right|\\
& \leq \lim \limits_{m\rightarrow \infty}\left|\frac{1}{19683}\right|^{pm} \text{max}\left\{\left|\mathcal{N}\left(\frac{x}{3^{pm}}\right)-n\left(\frac{x}{3^{pm}}\right)\right|, \left|n\left(\frac{x}{3^{pm}}\right)-\mathcal{N}^{\prime}\left(\frac{x}{3^{pm}}\right)\right|\right\}\\
&\leq \lim \limits_{m\rightarrow \infty}\lim \limits_{n\rightarrow \infty}\text{max}\Big\{\text{max}\Big\{\left|\frac{1}{19683}\right|^{p(l+m)+\frac{p-1}{2}}\zeta\left(\frac{x}{3^{p(l+m)+\frac{p+1}{2}}},\frac{x}{3^{p(l+m)+\frac{p+1}{2}}}\right):\\
&\qquad\qquad\qquad\qquad\qquad\qquad\qquad\qquad\qquad\qquad\qquad\qquad m\leq l\leq n+m\Big\}\Big\}=0
\end{align*}
for all $x\in \mathbb{A^{*}}$. This implies that $\mathcal{N}$ is distinct, which concludes the proof.
\end{proof}


From now on, we assume that $|2|<1$ for a non-Archimdean field $\mathbb A$. In the upcoming cosequences, we obtain the stability results of equation ({\ref{1.1}}) associated with the upper bound controlled by a fixed positive constant, sum of powers of norms, product of powers of norms and mixed product-sum of powers of norms by Theorem \ref{Thm 3.1}.

\begin{corollary}\label{Cor 3.3}
Let $\epsilon>0$ be a constant. If $n:\mathbb{A^{*}}\longrightarrow \mathbb{B}$ satisfies 
$\left|\Delta_1n(x,y)\right|\leq \epsilon$ for all $x,y\in \mathbb{A^{*}}$, then there exists a unique reciprocal-nonic mapping $\mathcal{N}:\mathbb{A^{*}}\longrightarrow \mathbb{B}$ satisfying \emph{(\ref{1.1})} and $\left|n(x)-\mathcal{N}(x)\right|\leq \epsilon$ for all $x\in \mathbb{A^{*}}$.
\end{corollary}
\begin{proof}
Considering $\zeta(x,y)=\epsilon$ in Theorem \ref{Thm 3.1} when $p=-1$, we obtain the requisite result.
$\hfill\square $
\end{proof}
\begin{corollary}\label{Cor 3.4}
Let $\epsilon>0$ and $q\neq -9$, be fixed constants. If $n:\mathbb{A^{*}}\longrightarrow \mathbb{B}$ satisfies 
$\left|\Delta_1n(x,y)\right|\leq \epsilon\left(\left|x\right|^q+\left|y\right|^q\right)$ for all $x,y\in \mathbb{A^{*}}$, then there exists a unique reciprocal-nonic mapping $\mathcal{N}:\mathbb{A^{*}}\longrightarrow \mathbb{B}$ satisfying \emph{(\ref{1.1})} and
\begin{equation*}
\left|n(x)-\mathcal{N}(x)\right|\leq
\begin{cases}
\frac{|2|\epsilon}{|3|^q}\left|x\right|^q, & \text{ ~$q>-9$}\\
|2|\epsilon|3|^9\left|x\right|^q, & \text{ ~$q<-9$}
\end{cases}
\end{equation*}
for all $x\in \mathbb{A^{*}}$.
\end{corollary}
\begin{proof}
Taking $\zeta(x,y)=\epsilon\left(\left|x\right|^q+\left|y\right|^q\right)$ in Theorem \ref{Thm 3.1}, the required result is achieved.
\end{proof}
\begin{corollary}\label{Cor 3.5}
Let $n:\mathbb{A^{*}}\longrightarrow \mathbb{B}$ be a mapping and let there exist real numbers $r,s, q=r+s \neq -9$ and $\epsilon> 0$ such that
$\left|\Delta_1n(x,y)\right|\leq \epsilon\left|x\right|^{r}\left|y\right|^{s}$ for all $x,y\in \mathbb{A^{*}}$. Then, there exists a unique reciprocal-nonic mapping $\mathcal{N}:\mathbb{A^{*}}\longrightarrow \mathbb{B}$ satisfying \emph{(\ref{1.1})} and 
\begin{equation*}
\left|n(x)-\mathcal{N}(x)\right|\leq
\begin{cases}
\frac{\epsilon}{|3|^{q}}\left|x\right|^{q}, & \text{ ~$q>-9$}\\
\epsilon|3|^9|\left|x\right|^{q}, & \text{ ~$q<-9$}
\end{cases}
\end{equation*}
for all $x\in \mathbb{A^{*}}$.
\end{corollary}
\begin{proof}
Letting $\zeta(x,y)=\epsilon\left|x\right|^{r}\left|y\right|^{s}$, for all $x,y\in \mathbb{A^{*}}$ in  Theorem \ref{Thm 3.1}, we attain the necessary result.
\end{proof}
\begin{corollary}\label{Cor 3.6}
Let $\epsilon> 0$ and $q\neq -9$ be real numbers, and $n:\mathbb{A^{*}}\longrightarrow \mathbb{B}$ be a mapping satisfying the functional inequality
\begin{equation*}
\left|\Delta_1n(x,y)\right|\leq \epsilon\left(\left|x\right|^{\frac{q}{2}}\left|y\right|^{\frac{q}{2}}+\left(\left|x\right|^{q}+\left|y\right|^{q}\right)\right)
\end{equation*}
for all $x,y\in \mathbb{A^{*}}$. Then, there exists a unique reciprocal-nonic mapping $\mathcal{N}:\mathbb{A^{*}}\longrightarrow \mathbb{B}$ satisfying \emph{(\ref{1.1})} and
\begin{equation*}
\left|n(x)-\mathcal{N}(u)\right|\leq
\begin{cases}
\frac{|3|\epsilon}{|3|^{q}}\left|x\right|^{q}, & \text{ ~$q>-9$}\\
|3|\epsilon|3|^9\left|x\right|^{a}, & \text{ ~$q<-9$}
\end{cases}
\end{equation*}
for all $x\in \mathbb{A^{*}}$.
\end{corollary}
\begin{proof}
Opting $\zeta(x,y)=\epsilon\left(\left|x\right|^{\frac{q}{2}}\left|y\right|^{\frac{q}{2}}+\left(\left|x\right|^{q}+\left|y\right|^{q}\right)\right)$ in Theorem \ref{Thm 3.1}, the result follows directly.
\end{proof}


The following theorem proves the stability result of the equation (\ref{1.2}). Even though the way of proving the result is similar to Theorem \ref{Thm 3.1}, for the purpose of comprehensiveness we provide the main skeleton of proof.

\begin{theorem}\label{Thm 3.6}
Let $p\in\{1,-1\}$ be fixed, and let $\xi:\mathbb{A^{*}}\times \mathbb{B^{*}}\longrightarrow [0,\infty)$ be a function such that
\begin{equation}\label{3.8}
\lim \limits_{k\rightarrow \infty}\left|\frac{1}{59049}\right|^{pk} \xi\left(\frac{x}{3^{pk+\frac{p+1}{2}}},\frac{y}{3^{pk+\frac{p+1}{2}}}\right)=0
\end{equation}
for all $x,y\in \mathbb{A^{*}}$. Suppose that $d:\mathbb{A^{*}}\longrightarrow \mathbb{B}$ is a mapping satisfying the inequality 
\begin{equation}\label{3.9}
\left|\Delta_2d(x,y)\right|\leq \xi(x,y)
\end{equation}
for all $x,y\in \mathbb{A^{*}}$. Then, there exists a unique reciprocal-decic mapping $\mathcal{D}:\mathbb{A^{*}}\longrightarrow \mathbb{B}$ satisfying \emph{(\ref{1.2})} and 
\begin{equation}\label{3.10}
\left|d(x)-\mathcal{D}(x)\right|\leq \textup{max}\left\{\left|\frac{1}{59049}\right|^{pl+\frac{p-1}{2}}\xi\left(\frac{x}{3^{pl+\frac{p+1}{2}}},\frac{x}{3^{pl+\frac{p+1}{2}}}\right): l\in \mathbb{N}\cup \{0\}\right\}
\end{equation}
for all $x\in \mathbb{A^{*}}$.
\end{theorem}
\begin{proof}
Letting $(x,y)$ as $(x,x)$ in (\ref{3.9}), we obtain
\begin{equation}\label{3.11}
\left|d(x)-\frac{1}{59049^{p}}d\left(\frac{x}{3^{p}}\right)\right|\leq |59049|^{\frac{|p-1|}{2}}\xi\left(\frac{x}{3^{\frac{p+1}{2}}},\frac{x}{3^{\frac{p+1}{2}}}\right)
\end{equation}
for all $x\in \mathbb{A^{*}}$.  Changing $x$ into $\frac{x}{3^{pk}}$ in (\ref{3.11}) and multiplying by $\left|\frac{1}{59049}\right|^{pk}$, one finds
\begin{align}\label{3.12}
& \left|\frac{1}{59049^{pk}}d\left(\frac{x}{3^{pk}}\right)-\frac{1}{59049^{p(k+1)}}d\left(\frac{x}{3^{p(k+1)}}\right)\right|\notag\\
&\qquad\qquad\qquad\qquad\qquad\qquad \leq \left|\frac{1}{59049}\right|^{pk+\frac{p-1}{2}}\xi\left(\frac{x}{3^{pk+\frac{p+1}{2}}},\frac{x}{3^{pk+\frac{p+1}{2}}}\right)
\end{align}
for all $x\in \mathbb{A^{*}}$. From the inequalities (\ref{3.8}) and (\ref{3.12}), we conclude that $\left\{\frac{1}{59049^{pk}}d\left(\frac{x}{3^{pk}}\right)\right\}$ is a Cauchy sequence. By the completeness of $\mathbb{B}$, there exists a mapping $\mathcal{D}:\mathbb{A^{*}}\longrightarrow \mathbb{B}$ so that 
\begin{equation}\label{3.13}
\mathcal{D}(x) = \lim \limits_{k\rightarrow \infty}\frac{1}{59049^{pk}}d\left(\frac{x}{3^{pk}}\right)
\end{equation}
for all $x\in \mathbb{A^{*}}$. The remaining part of the proof is alike Theorem \ref{Thm 3.1}.
\end{proof}


Using Theorem \ref{Thm 3.6}, we obtain the stability results of equation ({\ref{1.2}}) related with the upper bound controlled by a fixed positive constant, sum of powers of norms, product of powers of norms and mixed product-sum of powers of norms via the following corollaries.

\begin{corollary}\label{Cor 3.8}
Let $\theta>0$ be a constant, and let $d:\mathbb{A^{*}}\longrightarrow \mathbb{B}$ satisfies 
$\left|\Delta_2d(x,y)\right|\leq \theta$ for all $x,y\in \mathbb{A^{*}}$. Then, there exists a unique reciprocal-decic mapping $\mathcal{D}:\mathbb{A^{*}}\longrightarrow \mathbb{B}$ satisfying \emph{(\ref{1.2})} and
$\left|d(x)-\mathcal{D}(x)\right|\leq \theta$ for all $x\in \mathbb{A^{*}}$.
\end{corollary}
\begin{proof}
It is enough to put $\xi(x,y)=\theta$ in  Theorem \ref{Thm 3.6} in the case $p=-1$.
\end{proof}

\begin{corollary}\label{Cor 3.9}
Let $\theta> 0$ and $\alpha\neq -10$, be fixed constants. If $d:\mathbb{A^{*}}\longrightarrow \mathbb{B}$ satisfies 
$\left|\Delta_2d(x,y)\right|\leq \theta\left(\left|x\right|^{\alpha}+\left|y\right|^{\alpha}\right)$ for all $x,y\in \mathbb{A^{*}}$, then there exists a unique reciprocal-decic mapping $\mathcal{D}:\mathbb{A^{*}}\longrightarrow \mathbb{B}$ satisfying \emph{(\ref{1.2})} and
\begin{equation*}
\left|d(x)-\mathcal{D}(x)\right|\leq
\begin{cases}
\frac{|2|\theta}{|3|^{\alpha}}\left|x\right|^{\alpha}, & \text{ ~$\alpha>-10$}\\
|2|\alpha|3|^{10}\left|x\right|^{\alpha}, & \text{ ~$\alpha<-10$}
\end{cases}
\end{equation*}
for all $x\in \mathbb{A^{*}}$.
\end{corollary}
\begin{proof}
Allowing $\xi(x,y)=\theta\left(\left|x\right|^{\alpha}+\left|y\right|^{\alpha}\right)$, for all $x,y\in \mathbb{A^{*}}$ in Theorem \ref{Thm 3.6}, we attain the desired result. 
\end{proof}

\begin{corollary}\label{Cor 3.10}
Let $d:\mathbb{A^{*}}\longrightarrow \mathbb{B}$ be a mapping and let there exist real numbers $a,b, \alpha = a+b \neq -10$ and $\theta> 0$ such that
\begin{equation*}
\left|\Delta_2d(x,y)\right|\leq \theta\left|x\right|^{a}\left|y\right|^{b}
\end{equation*}
for all $x,y\in \mathbb{A^{*}}$. Then, there exists a unique reciprocal-decic mapping $\mathcal{D}:\mathbb{A^{*}}\longrightarrow \mathbb{B}$ satisfying \emph{(\ref{1.2})} and 
\begin{equation*}
\left|d(x)-\mathcal{N}(x)\right|\leq
\begin{cases}
\frac{\theta}{|3|^{\alpha}}\left|x\right|^{\alpha}, & \text{ ~$\alpha>-10$}\\
\lambda|3|^{10}\left|x\right|^{\alpha}, & \text{ ~$\alpha<-10$}
\end{cases}
\end{equation*}
for all $x\in \mathbb{A^{*}}$.
\end{corollary}
\begin{proof}
Choosing $\xi(x,y)=\theta\left|x\right|^{a}\left|y\right|^{b}$, for all $x,y\in \mathbb{A^{*}}$ in Theorem \ref{Thm 3.6}, the requisite result is achieved.
\end{proof}
\begin{corollary}\label{Cor 3.11}
Let $\theta> 0$ and $\alpha> -10$ be real numbers, and $d:\mathbb{A{*}}\longrightarrow \mathbb{B}$ be a mapping satisfying the functional inequality
\begin{equation*}
\left|\Delta_2d(x,y)\right|\leq \theta\left(\left|x\right|^{\frac{\alpha}{2}}\left|y\right|^{\frac{\alpha}{2}}+\left(\left|x\right|^{\alpha}+\left|y\right|^{\alpha}\right)\right)
\end{equation*}
for all $x,y\in \mathbb{A^{*}}$. Then, there exists a unique reciprocal-decic mapping $\mathcal{D}:\mathbb{A^{*}}\longrightarrow \mathbb{B}$ satisfying \emph{(\ref{1.2})} and
\begin{equation*}
\left|d(x)-\mathcal{D}(x)\right|\leq
\begin{cases}
\frac{|3|\theta}{|3|^{\alpha}}\left|x\right|^{\alpha}, & \text{ ~$\alpha>-10$}\\
|3|\theta|3|^{10}\left|x\right|^{\alpha}, & \text{ ~$\alpha<-10$}
\end{cases}
\end{equation*}
for every $x\in \mathbb{A^{*}}$.
\end{corollary}

\begin{proof}
It is simple to obtain the required result by selecting $$\xi(x,y)=\theta\left(\left|x\right|^{\frac{\alpha}{2}}\left|y\right|^{\frac{\alpha}{2}}+\left(\left|x\right|^{\alpha}+\left|y\right|^{\alpha}\right)\right)$$
in Theorem \ref{Thm 3.6}.
\end{proof}


\section{Proper examples}
We wind up this investigation with two proper examples. The famous counter-example provided by Gajda \cite{Gajda} enthused to prove the non-stability of the equations (\ref{1.1}) and (\ref{1.2}) for singular cases. In this section we illustrate that the stability results of functional equations (\ref{1.1}) and (\ref{1.2}) are not valid for $q=-9$ in Corollary \ref{Cor 3.4} and $\alpha=-10$ in Corollary \ref{Cor 3.9},  respectively.
\begin{example}\label{Ex 4.1} 
Let us consider the function
\begin{equation}\label{5.1}
\phi(x)=
\begin{cases}
\frac{k}{x^9}, & \text{for $x\in (1,\infty)$}\\
k, & \text{otherwise}
\end{cases}
\end{equation}
where $\phi:\mathbb{R^{*}}\longrightarrow \mathbb{R}$. Let $g:\mathbb{R^{*}}\longrightarrow \mathbb{R}$ be defined by
\begin{equation}\label{5.2}
g(x)=\sum_{m=0}^{\infty}19683^{-m}\phi(3^{-n}x)
\end{equation}
for all $x\in \mathbb{R}$.  Let the function $g:\mathbb{R^{*}}\longrightarrow \mathbb{R}$ defined in (\ref{5.2}) satisfies the functional inequality
\begin{equation}\label{5.3}
\left|\Delta_1g(x,y)\right|\leq \frac{29525~k}{9841}\left(\left|x\right|^{-9}+\left|y\right|^{-9}\right)
\end{equation}
for all $x,y\in \mathbb{R^{*}}$. We show that there do not exist a reciprocal-nonic mapping $\mathcal{N}:\mathbb{R^{*}}\longrightarrow \mathbb{R}$ and a constant $\alpha>0$ such that
\begin{equation}\label{5.4}
\left|g(x)-\mathcal{N}(x)\right|\leq \alpha\left|x\right|^{-9}
\end{equation}
for all $x\in \mathbb{R^{*}}$. For this, let us first prove that $g$ satisfies (\ref{5.3}). By computation, we have 
$$\left|g(x)\right|=\left|\sum_{m=0}^{\infty}19683^{-m}\phi(3^{-n}x)\right|\leq \sum_{m=0}^{\infty}\frac{k}{19683^{m}}=\frac{19683}{19682}k.$$
So, we see that $f$ is bounded by $\frac{19683~k}{19682}$ on $\mathbb{R}$. If $\left|x\right|^{-9}+\left|y\right|^{-9}\geq 1,$ then the left hand side of (\ref{5.3}) is less than $\frac{29525~k}{9841}.$ Now, suppose that $0<\left|x\right|^{-9}+\left|y\right|^{-9}<1$. Hence, there exists a positive integer $m$ such that
\begin{equation}\label{5.5}
\frac{1}{19683^{m+1}}\leq\left|x\right|^{-9}+\left|y\right|^{-9}<\frac{1}{19683^m}.
\end{equation}
Thus, the relation (\ref{5.5}) produces $19683^m\left(\left|x\right|^{-9}+\left|y\right|^{-9}\right)<1$, or equivalently; $19683^mx^{-9}<1, 19683^my^{-9}<1$. So, 
$ \frac{x^9}{19683^m}>1, \frac{y^9}{19683^m}>1$. The last inequalities imply that $\frac{x^9}{19683^{m-1}}>19683>1, \frac{y^9}{19683^{m-1}}>19683>1$ and consequently 
$$\frac{1}{3^{m-1}}(x)>1, \frac{1}{3^{m-1}}(y)>1, \frac{1}{3^{m-1}}(2x+y)>1, \frac{1}{3^{m-1}}(2x-y)>1.$$
Therefore, for each value of $m=0,1,2,\dots,n-1$, we obtain 
$$\frac{1}{3^n}(x)>1, \frac{1}{3^n}(y)>1, \frac{1}{3^n}(2x+y)>1, \frac{1}{3^n}(2x-y)>1.$$
and $\Delta_1g(3^{-n}x,3^{-n}y)=0$ for $m=0,1,2,\dots,n-1$. Using (\ref{5.1}) and the definition of $g$, we obtain
\begin{align*}
\left|\Delta_1g(x,y)\right| & \leq \sum_{m=n}^{\infty}\frac{k}{19683^m}+\sum_{m=n}^{\infty}\frac{k}{19683^m}+\frac{19684}{19683}\sum_{m=n}^{\infty}\frac{k}{19683^m}\\
&\leq \frac{59050~k}{19683}\frac{1}{19683^m}\left(1-\frac{1}{ 19683}\right)^{-1}\leq \frac{59050~k}{19682}\frac{1}{19683^m}\\
&\leq \frac{59050~k}{19682}\frac{1}{19683^{m+1}}\leq \frac{29525~k}{9841}\left(\left|x\right|^{-9}+\left|y\right|^{-9}\right)
\end{align*}
for all $x,y \in \mathbb{R^{*}}$. This means that the inequality (\ref{5.3}) holds. We claim that the reciprocal-nonic functional equation (\ref{1.1}) is not stable for $q=-9$ in Corollary \ref{Cor 3.4}. Assume that there exists a reciprocal-nonic mapping $\mathcal{N}:\mathbb{R^{*}}\longrightarrow \mathbb{R}$ satisfying (\ref{5.4}). So, we have
\begin{equation}\label{5.6}
|g(x)|\leq (\alpha +1)|x|^{-9}.
\end{equation}
However, we can choose a positive integer $m$ with $mk >\alpha +1$. If $x\in \left(1,3^{m-1}\right)$ then $3^{-n}x\in (1,\infty)$ for all $m=0,1,2,\dots,n-1$ and thus
\begin{equation*}
|g(x)| =\sum_{m=0}^{\infty}\frac{\phi(3^{-m}x)}{19683^m} \geq \sum_{m=0}^{n-1}\frac{\frac{19683^mk}{x^9}}{19683^m}=\frac{mk}{x^9}>(\alpha +1)x^{-9}
\end{equation*}
which contradicts (\ref{5.6}). Therefore, the reciprocal-nonic functional equation (\ref{1.1}) is not stable for $q=-9$ in Corollary \ref{Cor 3.4}.
\end{example}

In analogous to Example \ref{Ex 4.1}, we have the following result which acts as a counter-example for the fact that the functional equation (\ref{1.2}) is not stable for $\alpha=-10$ in Corollary \ref{Cor 3.9}. 
\begin{example}
Define the function  $\psi:\mathbb{R^{*}}\longrightarrow \mathbb{R}$ via
\begin{equation}\label{5.7}
\psi(x)=
\begin{cases}
\frac{c}{x^{10}} & \text{for $x\in (1,\infty)$}\\
c, & \text{otherwise}
\end{cases}
\end{equation}
Let $h:\mathbb{R^{*}}\longrightarrow \mathbb{R}$ be defined by
\begin{equation}\label{5.8}
h(x)=\sum_{m=0}^{\infty}59049^{-m}h(3^{-m}x)
\end{equation}
for all $m\in \mathbb{R}$. Assume that the function $h$ satisfies the functional inequality
\begin{equation}\label{5.9}
\left|\Delta_2h(x,y)\right|\leq \frac{44287~c}{14762}\left(\left|x\right|^{-10}+\left|y\right|^{-10}\right)
\end{equation}
for all $x,y\in \mathbb{R^{*}}$. Then, there do not exist a reciprocal-decic mapping $\mathcal{D}:\mathbb{R^{*}}\longrightarrow \mathbb{R}$ and a constant $\beta>0$ such that
\begin{equation}\label{5.10}
\left|h(x)-\mathcal{D}(x)\right|\leq \beta \left|x\right|^{-10}
\end{equation}
for all $x\in \mathbb{R^{*}}$. For this, we have
$$\left|h(x)\right|=\left|\sum_{m=0}^{\infty}59049^{-m}\psi(3^{-m}x)\right|\leq \sum_{m=0}^{\infty}\frac{c}{59049^{m}}=\frac{59049~c}{59048}.$$
Hence, we see that $h$ is bounded by $\frac{59049~c}{59048}$ on $\mathbb{R}$. If $\left|x\right|^{-10}+\left|y\right|^{-10}\geq 1,$ then the left hand side of (\ref{5.9}) is less than $\frac{44287~c}{14762}.$ Now, suppose that $0<\left|x\right|^{-10}+\left|y\right|^{-10}<1.$ Then, there exists a positive integer $m$ such that
\begin{equation}\label{5.11}
\frac{1}{59059^{m+1}}\leq\left|x\right|^{-10}+\left|y\right|^{-10}<\frac{1}{59049^m}.
\end{equation}
Similar to Example \ref{Ex 4.1}, the relation $\left|x\right|^{-10}+\left|y\right|^{-10}<\frac{1}{59049^m}$ implies
$$\frac{1}{3^{m-1}}(x)>1, \frac{1}{3^{m-1}}(y)>1, \frac{1}{3^{m-1}}(2x+y)>1, \frac{1}{3^{m-1}}(2x-y)>1.$$
Thus, for any $n=0,1,2,\dots,m-1$, we obtain 
$$\frac{1}{3^n}(x)>1, \frac{1}{3^n}(y)>1, \frac{1}{3^n}(2x+y)>1, \frac{1}{3^n}(2x-y)>1$$
and $\Delta_2h(3^{-n}x,3^{-n}y)=0$ for $n=0,1,2,\dots,m-1$. Applying (\ref{5.7}) and the definition of $h$, we find
\begin{align*}
\left|\Delta_2h(x,y)\right| & \leq \sum_{n=m}^{\infty}\frac{c}{59049^n}+\sum_{n=m}^{\infty}\frac{c}{59049^n}+\frac{59040}{59049}\sum_{n=k}^{\infty}\frac{c}{59049^n}\\
&\leq \frac{177148~c}{59048}\frac{1}{59049^m}\left(1-\frac{1}{59049}\right)^{-1}\\
&\leq \frac{1177148~c}{59048}\frac{1}{59049^m}\leq \frac{177148~c}{59048}\frac{1}{59049^{m+1}}\\
&\leq \frac{44287~c}{14762}\left(\left|x\right|^{-10}+\left|y\right|^{-10}\right)
\end{align*}
for all $x,y\in \mathbb{R^{*}}$. This shows that the inequality (\ref{5.9}) holds. Assume that there exists a reciprocal-decic mapping $\mathcal{D}:\mathbb{R^{*}}\longrightarrow \mathbb{R}$ satisfying (\ref{5.10}). Hence
\begin{equation}\label{5.12}
|h(x)|\leq (\beta +1)|x|^{-10}.
\end{equation}
On the other hand, we can choose a positive integer $k$ with $kc >\beta +1$. If $x\in \left(1,3^{k-1}\right)$ then $3^{-n}x\in (1,\infty)$ for all $n=0,1,2,\dots,k-1$ and so
\begin{equation*}
|h(x)| =\sum_{n=0}^{\infty}\frac{\psi(3^{-n}x)}{59049^n} \geq \sum_{n=0}^{k-1}\frac{\frac{59049^n~c}{x^{10}}}{59049^n}=\frac{kc}{x^{10}}>(\beta +1)x^{-10}
\end{equation*}
which contradicts (\ref{5.12}). Therefore, the reciprocal-decic functional equation (\ref{1.2}) is not stable for $\alpha=-10$ in Corollary \ref{Cor 3.9}.
\end{example}

\bibliographystyle{amsplain}

\end{document}